\documentclass{amsart}

\usepackage{amsmath,amsthm,amstext,amssymb,bbm}
\usepackage{hyperref}
\usepackage{accents}
\usepackage[utf8]{inputenc}
\usepackage{enumerate,enumitem}
\usepackage{amscd,  amsfonts,  amsbsy,  mathrsfs, upgreek, mathtools, stmaryrd, bbm}

\theoremstyle{plain}
\newtheorem{theorem}{Theorem}[section]
\newtheorem{lemma}[theorem]{Lemma}

\newtheorem*{observation*}{Observation}

\newtheorem*{claim*}{Claim}
\newtheorem*{subclaim*}{Subclaim}
\theoremstyle{definition}
\newtheorem{definition}[theorem]{Definition}

\newcommand{\crit}[1]{{{\rm{crit}}\left({#1}\right)}}

\newcommand{\ran}[1]{{{\rm{ran}}(#1)}}

\newcommand{\POT}[1]{{\mathcal{P}}({#1})}

\newcommand{\map}[3]{{#1}:{#2}\longrightarrow{#3}}

\newcommand{\Set}[2]{\{{#1}~\vert~{#2}\}}
\newcommand{\seq}[2]{\langle{#1}~\vert~{#2}\rangle}

\newcommand{\Poti}[2]{{\mathcal{P}}_{#2}(#1)}
\newcommand{\HH}[1]{{\rm{H}}(#1)}

\newcommand{\Add}[2]{{\rm{Add}}({#1},{#2})}

\newcommand{\id}{{\rm{id}}}

\newcommand{\On}{{\rm{Ord}}}
\newcommand{\LL}{{\rm{L}}}
\newcommand{\HOD}{{\rm{HOD}}}
\newcommand{\ZF}{{\rm{ZF}}}
\newcommand{\ZFC}{{\rm{ZFC}}}

\newcommand{\PPP}{{\mathbb{P}}}
\newcommand{\QQQ}{{\mathbb{Q}}}

\newcommand{\VV}{{\rm{V}}}

\newcommand{\scrL}{{\mathscr{L}}}

\newcommand{\calA}{\mathcal{A}}

\newcommand{\calF}{\mathcal{F}}

\newcommand{\calL}{\mathcal{L}}

\newcommand{\calU}{\mathcal{U}}

 \newcommand{\calS}{\mathcal{S}}

\makeatletter
\@namedef{subjclassname@2020}{%
  \textup{2020} Mathematics Subject Classification}
\makeatother

\title{Strong unfoldability, shrewdness and combinatorial consequences}

\author{Philipp L\"ucke}
\address{Institut de Matem\`{a}tica, Universitat de Barcelona.  Gran via de les Corts Catalanes 585, 08007 Barcelona, Spain.}
\email{philipp.luecke@ub.edu}


\thanks{This project has received funding from the European Union’s Horizon 2020 research and innovation programme under the Marie Sk{\l}odowska-Curie grant agreement No 842082 (Project \emph{SAIFIA: Strong Axioms of Infinity -- Frameworks, Interactions and Applications}). 
 The author would like to thank Gunter Fuchs and Victoria Gitman for a discussion that motivated a significant part of the work presented in this note. 
 In addition, the author would like to thank the anonymous referee for numerous suggestions and corrections.}

\begin{document}

\begin{abstract}
 We show that the notions of \emph{strongly unfoldable cardinals}, introduced by Villaveces in his model-theoretic studies of models of set theory, and \emph{shrewd cardinals}, introduced by Rathjen in a proof-theoretic context,  coincide. 
 We then proceed by using ideas from the proof of this equivalence to establish the existence of \emph{ordinal anticipating Laver functions} for strong unfoldability. 
  With the help of  these functions, we  show that the principle $\Diamond_\kappa(\mathrm{Reg})$ holds at every strongly unfoldable cardinal $\kappa$ with the property that there exists a subset $z$ of $\kappa$ such that every subset of $\kappa$ is ordinal definable from $z$.  
 While a result of D\v{z}amonja and Hamkins  shows that $\Diamond_\kappa(\mathrm{Reg})$ can consistently fail at a strongly unfoldable cardinal $\kappa$, 
 this implication can be used to prove that various canonical extensions of the axioms of $\ZFC$ are either compatible with the assumption that $\Diamond_\kappa(\mathrm{Reg})$ holds at every strongly unfoldable cardinal $\kappa$ or outright imply this statement.  
 Finally, we will also use our methods to contribute to the study of strong chain conditions of partials orders and their productivity. 
\end{abstract}

\maketitle


\section{Introduction}

This note contributes to the study of \emph{strongly unfoldability}, a large cardinal notion introduced by Villaveces in his investigation of chains of end elementary extensions of models of set theory in \cite{MR1649079}.

\begin{definition}[Villaveces]
 An inaccessible cardinal $\kappa$ is \emph{strongly unfoldable} if for every ordinal $\lambda$ and every transitive $\ZF^-$-model\footnote{A short argument shows that the restriction of this property to models of $\ZFC^-$ is equivalent to the original property. The proof of Theorem \ref{theorem:Equivalent} below directly shows that strong unfoldability is already equivalent to the restriction of the defining property to elementary submodels of $\HH{\kappa^+}$.} $M$ of cardinality $\kappa$ with $\kappa\in M$ and ${}^{{<}\kappa}M\subseteq M$, there is a transitive set $N$ with $\VV_\lambda\subseteq N$ and an elementary embedding $\map{j}{M}{N}$ with $\crit{j}=\kappa$ and $j(\kappa)\geq\lambda$. 
\end{definition}

This large cardinal notion possesses two features that can be combined in a profound and fruitful way to make the study of these cardinals very appealing. 
First, strongly unfoldable cardinals are located relatively low in the hierarchy of large cardinals. 
 More specifically, results in {\cite[Section 2]{MR2279655}}  show that their consistency strength is strictly between total indescribability and subtleness.  
 %
 Moreover, {\cite[Theorem 2.1]{MR1649079}} shows that strongly unfoldable cardinals can exist in G\"odel's constructible universe $\LL$. 
 Second, strongly unfoldable cardinals can be seen as \emph{miniature} versions of supercompact cardinals and, for several important results utilizing supercompactness, analogous statements can be shown to hold for strong unfoldability. 
 For example, results of Hamkins and Johnstone in \cite{MR2470843} show that the consistency of the restriction of the \emph{Proper Forcing Axiom} to proper forcing notions that preserve either $\aleph_2$ or $\aleph_3$ can be established by forcing over a model containing a strongly unfoldable cardinal and therefore this forcing  axiom can hold in forcing extensions of $\LL$.

 In this note, we continue this line of research and provide further examples in which constructions based on strong large cardinal assumptions can be adapted to strongly unfoldable cardinals. 
 The starting point of our work is the  following large cardinal property, introduced by  Rathjen in \cite{MR1369172} in his research on ordinal analysis in proof theory.

 \begin{definition}[Rathjen]\label{definition:ShrewdCardinals}
   A cardinal $\kappa$ is \emph{shrewd} if for every formula $\Phi(v_0,v_1)$ in the language $\calL_\in$ of set theory, every ordinal $\gamma>\kappa$ and every  $A\subseteq\VV_\kappa$ with the property that $\Phi(A,\kappa)$ holds in $\VV_\gamma$, there exist ordinals $\alpha<\beta<\kappa$ such that $\Phi(A\cap\VV_\alpha,\alpha)$ holds in $\VV_\beta$. 
\end{definition}

  A classical result of Magidor in \cite{MR0295904} shows that a cardinal $\kappa$ is supercompact if and only if for every cardinal $\theta>\kappa$, there exists a cardinal $\bar{\theta}<\kappa$ and a non-trivial elementary embedding $\map{j}{\HH{\bar{\theta}}}{\HH{\theta}}$ with $j(\crit{j})=\kappa$.
  The results of {\cite[Section 2]{SRminus}} show that an analogous characterization of shrewd cardinals can be obtained by replacing the domains of the relevant embeddings with elementary submodels of small cardinality. 
  In Section \ref{section:equiv}, we will use this characterization to show that the two large cardinal notions discussed above are equivalent.

\begin{theorem}\label{theorem:Equivalent}
 A cardinal is strongly unfoldable if and only if it is shrewd. 
\end{theorem}

Recall that, given an uncountable regular cardinal $\kappa$ and a cardinal $\lambda\geq\kappa$, Miyamoto defined the cardinal $\kappa$ to be \emph{$H_\lambda$-reflecting} (see \cite{MR1789737} and \cite{MiyFast}), 
if for every $\calL_\in$-formula $\varphi(v)$ and all $z\in\HH{\lambda}$ with the property that there exists a cardinal $\theta$ with $z\in\HH{\theta}$ and $\HH{\theta}\models\varphi(z)$, the collection $\Poti{\HH{\lambda}}{\kappa}$ contains a stationary subset consisting of elementary submodels $X$ of $\HH{\lambda}$ 
with $z\in X$ and 
the property that when $\map{\pi}{X}{M}$ denotes the  transitive collapse of $X$, then there exists a cardinal $\bar{\theta}<\kappa$ with $\pi(z)\in\HH{\bar{\theta}}$ and $\HH{\bar{\theta}}\models\varphi(\pi(z))$.  
It is easy to see that every $H_{\kappa^+}$-reflecting cardinal is shrewd and results of Miyamoto in {\cite[Section 1]{MR1789737}} show that strong unfoldability is equivalent to $H_{\kappa^+}$-reflection. 
The equivalence provided by Theorem \ref{theorem:Equivalent} now shows that it is possible to substantially weaken the stationarity assumption in Miyamoto's definition and still obtain the same property.  
%

In the case of supercompact cardinals $\kappa$, several important properties of these cardinals can be more easily derived using the embeddings $\map{j}{\HH{\bar{\theta}}}{\HH{\theta}}$ provided by Magidor's results in \cite{MR0295904} 
 %
 and the absoluteness of various set-theoretic statements between the set-theoretic universe $\VV$ and the sets $\HH{\kappa}$, $\HH{\theta}$ and $\HH{\bar{\theta}}$.
 %
 It now turns out that the correctness properties of the domain models of the embeddings used in the proof of Theorem \ref{theorem:Equivalent}  still suffice to  carry out some of these arguments and thereby derive interesting consequences of strong unfoldability. 
  In the last two sections of this paper, we will present two examples of this approach. 

 In Section \ref{section:laver}, 
 we study the validity of strong guessing principles at strongly unfoldable cardinals. 
 A short argument (see {\cite[Observation 1]{MR2279655}}) shows that the principle $\Diamond_\kappa(\mathrm{Reg})$ holds at every measurable cardinal $\kappa$, i.e. for every measurable cardinal $\kappa$, there exists a sequence $\seq{A_\alpha}{\alpha<\kappa}$ with the property that for every subset $A$ of $\kappa$, the set of regular cardinals $\alpha<\kappa$ with $A\cap\alpha=A_\alpha$ is stationary in $\kappa$. Results of Jensen and Kunen in \cite{JensenKunen1969:Ineffable} show that the same conclusion holds for subtle cardinals. 
 In contrast,  Hauser \cite{MR1164732} generalized earlier work of Woodin by showing that for all natural numbers $m$ and $n$, the principle $\Diamond_\kappa(\mathrm{Reg})$ can consistently fail at a $\Pi^m_n$-indescribable cardinal $\kappa$. 
 D\v{z}amonja and Hamkins improved this result in \cite{MR2279655} by showing that $\Diamond_\kappa(\mathrm{Reg})$ can consistently fail at a strongly unfoldable cardinal $\kappa$.  
  In Section \ref{section:laver}, we will show that, in many canonical models of set theory, the principle $\Diamond_\kappa(\mathrm{Reg})$ holds at all strongly unfoldable cardinals $\kappa$ by proving following result. 

\begin{theorem}\label{theorem:HODdiamonds}
  If $\kappa$ is a strongly unfoldable cardinal with $\POT{\kappa}\subseteq\HOD_z$  for some $z\subseteq\kappa$, then $\Diamond_\kappa(\mathrm{Reg})$ holds. 
\end{theorem}

 The results of \cite{MR2279655} already show that the above statement cannot be extended to parameters that are subsets of $\kappa^+$, because, by starting with a strongly unfoldable cardinal $\kappa$ in $\LL$, they can be used to produce a cardinal preserving forcing extension of $\LL$ of the form $\LL[A]$ with $A\subseteq\kappa^+$, in which $\kappa$ is still strongly unfoldable and $\Diamond_\kappa(\mathrm{Reg})$ fails. 
 Moreover, since the assumption $\VV=\HOD$ can be forced to hold by class forcings with arbitrary high degrees of closure and the total indescribability of a given cardinal is preserved by sufficiently closed forcings, the results of \cite{MR2279655} show that  the assumption $\VV=\HOD$ does not imply that $\Diamond_\kappa(\mathrm{Reg})$ holds at all totally indescribable cardinals $\kappa$. 
 %
 %
 Finally, it is  easy to show that the converse of the implication of Theorem \ref{theorem:HODdiamonds} can consistently fail. 
 If $\kappa$ is a strongly unfoldable cardinal whose strong unfoldability is indestructible under forcing with the partial order $\Add{\kappa}{\kappa^+}$ that adds $\kappa^+$-many Cohen subsets to $\kappa$ (see \cite{MR2467213}) and $G$ is $\Add{\kappa}{\kappa^+}$-generic over the ground model $\VV$, then standard arguments  show that, in $\VV[G]$, the principle $\Diamond_\kappa(\mathrm{Reg})$ holds and we have $\POT{\kappa}\nsubseteq\HOD_z$ for all subsets $z$ of $\kappa$.

 A topic closely related to the validity of $\Diamond$-principles is the generalization of Laver's classical result on the existence of \emph{Laver functions} for supercompact cardinals (see \cite{MR472529}) to other large cardinal notions. 
  Since the existence of such functions directly yields a $\Diamond_\kappa(\mathrm{Reg})$-sequence, the main result of \cite{MR2279655} shows that the non-existence of Laver functions for strong unfoldability is consistent. 
  In Section \ref{section:laver}, we will prove that \emph{ordinal anticipating Laver functions for strong unfoldability} (see {\cite[Definition 16]{MR2279655}}) always exist. 
  More specifically, we construct a class function $\map{\scrL}{\On}{\On}$ with the property that for every strongly unfoldable cardinal $\kappa$, we have $\scrL[\kappa]\subseteq\kappa$ and for all ordinals $\alpha$ and $\lambda$, and every transitive $\ZF^-$-model $M$ of cardinality $\kappa$ with ${\scrL\restriction\kappa}\in M$ and ${}^{{<}\kappa}M\subseteq M$, there exists a transitive set $N$ with $\VV_\lambda\subseteq N$
 and an elementary embedding $\map{j}{M}{N}$ 
 with $\crit{j}=\kappa$, $j(\kappa)\geq\lambda$ and $j({\scrL\restriction\kappa})(\kappa)=\alpha$.  
  This result strengthens {\cite[Theorem 10]{MR2467213}} that proves the existence of a class function $\map{\mathscr{F}}{\On}{\On}$ whose restrictions to strongly unfoldable cardinals possess the \emph{Menas property} for strong unfoldability (see {\cite[Section 3]{MR2467213}}).\footnote{After a first version of this paper was submitted, the author noticed that the statement of Theorem \ref{theorem:OrdinalAntiLaver}.(ii) below can also be easily  derived from arguments made in the proof of  {\cite[Theorem 10]{MR2467213}} by slightly modifying the constructed class function $F$ and placing transitive $\ZF^-$-models in transitive $\ZFC^-$-models of the same cardinality.  
  While the arguments in \cite{MR2467213} rely on the degrees of failures of strong unfoldability, the proof presented below is motivated by related constructions for stronger large cardinal notions (see, for example, {\cite[Observation 1]{MR2279655}} and \cite{MR472529}) that diagonalize against possible counterexamples. 
  }    
  %
  %
  Finally, these constructions allow us to show that the assumption of Theorem \ref{theorem:HODdiamonds} implies the existence of a full Laver function for strong unfoldability (as defined in {\cite[Section 1]{Hamkins:LaverDiamond}} and {\cite[Section 1]{MR1789737}}). 
 This implication will then directly provide a proof of Theorem \ref{theorem:HODdiamonds}.

In Section \ref{section:FilterLayered}, we will provide another application of the techniques developed in this paper by utilizing them in the study of strong chain conditions of partial orders and their productivity. 
 The results of \cite{MR3620068} show that the weak compactness of a cardinal $\kappa$ can be characterized through the equivalence of the $\kappa$-chain condition to a combinatorial property of partial orders that is, in general, stronger than the chain condition. 
 This work is motivated by a question of Todor\v{c}evi\'{c} (see {\cite[Question 8.4.27]{MR2355670}}), asking whether all regular cardinals $\kappa>\aleph_1$ with the property that the $\kappa$-chain condition is equivalent to the \emph{productive $\kappa$-chain condition}\footnote{Remember that, given an uncountable regular $\kappa$, a partial order $\PPP$ has the \emph{productive $\kappa$-chain condition} if for every partial order $\QQQ$ satisfying the $\kappa$-chain condition, the product $\PPP\times\QQQ$ also satisfies the $\kappa$-chain condition. An easy argument shows that the $\kappa$-Knaster property implies the productive $\kappa$-chain condition.} are weakly compact.\footnote{If $\kappa$ is weakly compact, then the $\kappa$-chain condition is equivalent to the $\kappa$-Knaster property (see {\cite[Proposition 1.1]{MR3620068}}). A series of deep results shows that, for regular cardinals $\kappa>\aleph_1$, many important consequences of weak compactness can be derived from the productivity of  the $\kappa$-chain condition. An overview of these results can be found in \cite{MR3271280}.}

  The  theory developed in \cite{MR3620068} suggests that large cardinal properties much stronger than weak compactness ({e.g.}  supercompactness) can be characterized in similar ways and it already provides a strong chain condition, called \emph{filter layeredness}, as a natural candidate for such generalizations. 
 In \cite{MR4082998}, Cody proved that such characterizations are not provable by showing that, consistently, filter layeredness can be equivalent to the $\kappa$-chain condition at the least weakly compact cardinal $\kappa$.  
  Since Cody used large cardinal assumptions close to supercompactness in his consistency proof, the question whether the equivalence between these two chain conditions has very high consistency strength remained open. 
  In Section \ref{section:FilterLayered}, we will derive a negative answer to this questions by showing that filter layeredness is equivalent to the $\kappa$-chain condition whenever $\kappa$ is a strongly unfoldable cardinal. 


\section{Shrewdness is equivalent to strong unfoldability}\label{section:equiv}

In this section, we prove Theorem \ref{theorem:Equivalent}. Our arguments are based on the Magidor-style embedding characterization of shrewd cardinals developed in {\cite[Section 2]{SRminus}}. 
 The following lemma slightly strengthens {\cite[Lemma 2.1]{SRminus}} by  showing that, in the formulation of these embedding characterizations, we can also assume that the domain models of our embeddings are closed under sequences of length smaller than the critical point of the embedding. 
 Since this closure property will be crucial for later arguments, we present a complete proof of this stronger result.

\begin{lemma}\label{lemma:ShrewdEmbeddings}
 The following statements are equivalent for every  cardinal $\kappa$: 
 \begin{enumerate}[leftmargin=0.9cm]
  \item $\kappa$ is a shrewd cardinal.  
  
    \item For all cardinals $\theta>\kappa$ and all $z\in\HH{\theta}$, there exist 
  cardinals $\bar{\kappa}<\bar{\theta}<\kappa$, 
  an elementary submodel $X$ of $\HH{\bar{\theta}}$  with  $(\bar{\kappa}+1)\cup{}^{{<}\bar{\kappa}}X\subseteq X$  
  and an elementary embedding $\map{j}{X}{\HH{\theta}}$ with $j\restriction\bar{\kappa}=\id_{\bar{\kappa}}$,  $j(\bar{\kappa})=\kappa$ and $z\in\ran{j}$. 
 \end{enumerate}
\end{lemma}

\begin{proof}
 Assume that (i) holds. 
  Then we can find an $\calL_\in$-formula $\Phi(v_0,v_1)$ with the property that for every ordinal $\gamma$ and all $A,\delta\in\VV_\gamma$, the statement $\Phi(A,\delta)$ holds in $\VV_\gamma$ if and only if $\gamma$ is a limit ordinal and the following statements hold in $\VV_\gamma$: 
  \begin{enumerate}
     \item[(a)] There exist unboundedly many strong limit cardinals. 
   
   \item[(b)] $\delta$ is an inaccessible cardinal. 
   
   \item[(c)] There exists a cardinal $\theta>\delta$, a subset $X$ of the class $\HH{\theta}$ and a bijection $\map{b}{\delta}{X}$ such that the following statements hold: 
    \begin{itemize}    
     \item The class $\HH{\theta}$ is a set. 
    
     \item $(\delta+1)\cup{}^{{<}\delta}X\subseteq X$.
     
     \item $b(0)=\delta$ and $b(\omega\cdot(1+\alpha))=\alpha$ for all $\alpha<\delta$. 
    
     
     \item Given $\alpha_0,\ldots,\alpha_{n-1}<\delta$ and $a\in\mathsf{Fml}$ represents a formula with $n$ free variables, we have 
         \begin{equation*}
          \begin{split}
            \langle a,\alpha_0,\ldots,\alpha_{n-1}\rangle\in A ~ & \Longleftrightarrow ~ \mathsf{Sat}(X,\langle b(\alpha_0),\ldots,b(\alpha_{n-1})\rangle,a) \\
            & \Longleftrightarrow ~ \mathsf{Sat}(\HH{\theta},\langle b(\alpha_0),\ldots,b(\alpha_{n-1})\rangle,a),  
           \end{split}
         \end{equation*}
      where $\mathsf{Fml}$ denotes the set of  formalized $\calL_\in$-formulas and $\mathsf{Sat}$ denotes the  formalized satisfaction relation for $\calL_\in$-formulas.\footnote{Details on these formalizations can be found in {\cite[Section I.9]{MR750828}}. Note that the classes $\mathsf{Fml}$ and $\mathsf{Sat}$ are defined by $\Sigma_1$-formulas without parameters.}  
    \end{itemize}
  \end{enumerate}
 
 Fix a cardinal $\theta>\kappa$ and some $z\in\HH{\theta}$.
 Pick an elementary submodel $Y$ of $\HH{\theta}$ of cardinality $\kappa$ with $\kappa\cup\{\kappa,z\}\cup{}^{{<}\kappa}Y\subseteq Y$ and a bijection $\map{b}{\kappa}{Y}$ with  $b(0)=\kappa$, $b(1)=z$ and $b(\omega\cdot(1+\alpha))=\alpha$ for all $\alpha<\kappa$. 
 Define $A$ to be the set of all tuples $\langle a,\alpha_0,\ldots,\alpha_{n-1}\rangle$ with the property that $a\in\mathsf{Fml}$ represents an  $\calL_\in$-formula with $n$ free variables,  $\alpha_0,\ldots,\alpha_{n-1}<\kappa$ and $\mathsf{Sat}(Y,\langle b(\alpha_0),\ldots,b(\alpha_{n-1})\rangle,a)$ holds.  
 Now, pick a cardinal $\lambda>\theta$ that is a limit of fixed points of the $\beth$-functions. 
 %
 Then the statement $\Phi(A,\kappa)$ holds in $\VV_\lambda$ and the shrewdness of $\kappa$ yields ordinals $\bar{\kappa}<\beta<\kappa$ with the property that $\Phi(A\cap\VV_{\bar{\kappa}},\bar{\kappa})$ holds in $\VV_\beta$. 
 Then $\beta$ is a limit ordinal and, since $\VV_\beta$ is correct about both strong limit and regular cardinals less than $\beta$, it follows that $\bar{\kappa}$ is an inaccessible cardinal and $\beta$ is a strong limit cardinal.  
 Similar absoluteness and correctness considerations show that there is a cardinal $\bar{\kappa}<\bar{\theta}<\beta$ with $\HH{\bar{\theta}}\in\VV_\beta$, a subset $X$ of $\HH{\bar{\theta}}$ and a bijection $\map{\bar{b}}{\bar{\kappa}}{X}$ such that  $(\bar{\kappa}+1)\cup{}^{{<}\bar{\kappa}}X\subseteq X$, $\bar{b}(0)=\bar{\kappa}$, $\bar{b}(\omega\cdot(1+\alpha))=\alpha$ for all $\alpha<\bar{\kappa}$ and 
          \begin{equation*}
          \begin{split}
            \langle a,\alpha_0,\ldots,\alpha_{n-1}\rangle\in A\cap\VV_{\bar{\kappa}} ~ & \Longleftrightarrow ~ \mathsf{Sat}(X,\langle \bar{b}(\alpha_0),\ldots,\bar{b}(\alpha_{n-1})\rangle,a) \\
            & \Longleftrightarrow ~ \mathsf{Sat}(\HH{\bar{\theta}},\langle\bar{b}(\alpha_0),\ldots,\bar{b}(\alpha_{n-1})\rangle,a)  
           \end{split}
         \end{equation*}
         for all $\alpha_0,\ldots,\alpha_{n-1}<\bar{\kappa}$ and every $a\in\mathsf{Fml}$ that represents a formula with $n$ free variables. 
    This directly implies that $X$ is an elementary submodel of $\HH{\bar{\theta}}$. 
    Moreover, given an $\calL_\in$-formula $\varphi(v_0,\ldots,v_{n-1})$ and  $\alpha_0,\ldots,\alpha_{n-1}<\bar{\kappa}$, we  have

    \begin{equation*}
     \begin{split}
      X\models\varphi(\bar{b}(\alpha_0),\ldots,\bar{b}(\alpha_{n-1})) ~ & \Longleftrightarrow ~   \langle \ulcorner\varphi\urcorner,\alpha_0,\ldots,\alpha_{n-1}\rangle\in A\cap\VV_{\bar{\kappa}} \\ 
      ~ \Longleftrightarrow  ~ \langle \ulcorner\varphi\urcorner,\alpha_0,\ldots,\alpha_{n-1}\rangle\in A ~ & \Longleftrightarrow  ~  \HH{\theta}\models\varphi(b(\alpha_0),\ldots,b(\alpha_{n-1})),   
   \end{split}
  \end{equation*}
  where $\ulcorner\varphi\urcorner$ denotes the canonical element of $\mathsf{Fml}$ that represents the formula $\varphi$. 
 This allows us to conclude that 
   the map $\map{j=b\circ\bar{b}^{{-}1}}{X}{\HH{\theta}}$ is an elementary embedding with $j\restriction\bar{\kappa}=\id_{\bar{\kappa}}$, $j(\bar{\kappa})=\kappa$ and $z\in\ran{j}$. In particular, we know that (ii) holds in this case.

  Now, assume that (ii) holds. Fix an $\calL_\in$-formula $\Phi(v_0,v_1)$, an ordinal $\gamma>\kappa$ and a subset $A$ of $\VV_\kappa$ such that $\Phi(A,\kappa)$ holds in $\VV_\gamma$.  
  Pick a   cardinal $\theta>\gamma$ that is a limit of fixed points of the $\beth$-function.
  Then $\VV_\gamma\in\HH{\theta}$ and we can now use our assumption to find cardinals $\bar{\kappa}<\bar{\theta}<\kappa$, an elementary submodel $X$ of $\HH{\bar{\theta}}$ and  an elementary embedding $\map{j}{X}{\HH{\theta}}$ such that   $(\bar{\kappa}+1)\subseteq{}^{{<}\bar{\kappa}}X\subseteq X$, $j\restriction\bar{\kappa}=\id_{\bar{\kappa}}$,  $j(\bar{\kappa})=\kappa$ and $A,\gamma\in\ran{j}$. 
  Pick $\beta\in X\cap(\bar{\kappa},\bar{\theta})$ with $j(\beta)=\gamma$. Since elementarity implies that $\bar{\theta}$ is also a limit of fixed points of the $\beth$-function, we know that $\VV_\beta\in\HH{\bar{\theta}}$. 
  The closure properties of $X$ now imply that $\bar{\kappa}$ is an inaccessible cardinal and we can use the elementarity of $j$ to show that $\kappa$ is also inaccessible. 
  In particular, we have $\VV_{\bar{\kappa}}\subseteq X$ and $j\restriction\VV_{\bar{\kappa}}=\id_{\VV_{\bar{\kappa}}}$. 
  But this  shows that $A\cap\VV_{\bar{\kappa}}\in X$ and $j(A\cap\VV_{\bar{\kappa}})=A$. 
  Elementarity now allows us to conclude that $\Phi(A\cap\VV_{\bar{\kappa}},\bar{\kappa})$ holds in $\VV_\beta$.  
\end{proof}

 We are now ready to show that Villaveces' notion of strongly unfoldable cardinals coincides with Rathjen's notion of shrewd cardinals.

\begin{proof}[Proof of Theorem \ref{theorem:Equivalent}]
 First, assume that $\kappa$ is a strongly unfoldable cardinal.
  Fix an $\calL_\in$-formula $\Phi(v_0,v_1)$, an ordinal  $\gamma>\kappa$  and  a subset $A$  of $\VV_\kappa$ with the property  that $\Phi(A,\kappa)$ holds in $\VV_\gamma$.  
 Using the inaccessibility of $\kappa$, we can find an elementary submodel $M$ of $\HH{\kappa^+}$ of cardinality $\kappa$ with $\VV_\kappa\cup\{\kappa,A\}\cup{}^{{<}\kappa}M\subseteq M$. 
 By our assumptions on $\kappa$, there exists a transitive set $N$ with $\VV_{\gamma+1}\subseteq N$ and an elementary embedding $\map{j}{M}{N}$ with $\crit{j}=\kappa$ and $j(\kappa)\geq\gamma+1$. 
  Since $j(A)\cap\VV_\kappa=A$, we know that, in $N$, there exist ordinals $\delta<\varepsilon<j(\kappa)$ with the property that $\Phi(j(A)\cap\VV_\delta,\delta)$ holds in $\VV_\varepsilon$. 
 Using the elementarity of $j$ and the fact that $\VV_\kappa$ is a subset of $M$, we can conclude that there exist ordinals $\alpha<\beta<\kappa$ such that $\Phi(A\cap\VV_\alpha,\alpha)$ holds in $\VV_\beta$.

  Next, assume that $\kappa$ is a shrewd cardinal,  $\lambda>\kappa$ is an ordinal and $M$ is a transitive $\ZF^-$-model  of cardinality $\kappa$ with $\kappa\in M$ and ${}^{{<}\kappa}M\subseteq M$. 
  Since shrewdness implies total indescribability (see {\cite[Section 2]{zbMATH02168085}} as well as {\cite[Chapter 9, Lemma 4.2]{MR3408725}}), we know that $\kappa$ is inaccessible. 
  Pick a cardinal $\theta>\lambda$ that is a limit of fixed points of the $\beth$-function 
  %
  and   use Lemma \ref{lemma:ShrewdEmbeddings} to find    cardinals $\bar{\kappa}<\bar{\theta}<\kappa$, 
  an elementary submodel $X$ of $\HH{\bar{\theta}}$   %
  and an elementary embedding $\map{j}{X}{\HH{\theta}}$ such that  $\bar{\kappa}+1\subseteq X$,    $j\restriction\bar{\kappa}=\id_{\bar{\kappa}}$,  $j(\bar{\kappa})=\kappa$ and $\lambda,M\in\ran{j}$. 
  Then $\bar{\theta}$ is a limit of fixed points of the $\beth$-function, because elementarity and our choice of $\theta$ ensure that the Power Set Axiom holds in $\HH{\bar{\theta}}$ and hence $\HH{\bar{\theta}}$ correctly computes the $\beth$-function up to $\bar{\theta}$. 
  In particular, we have   $\VV_\xi\in\HH{\bar{\theta}}$ for every $\xi<\bar{\theta}$. 
  Pick $\bar{\lambda}$ and $\bar{M}$ in $X$ satisfying $j(\bar{\lambda})=\lambda$ and $j(\bar{M})=M$. 
  We then have $\bar{\lambda}<\bar{\theta}<\kappa$, $\VV_\kappa\subseteq M$, $\bar{M}\subseteq X$ and $\VV_{\bar{\lambda}}\in X$. 
  This shows that the model $M$ and the embedding $j\restriction\bar{M}$ witness that there exists a transitive set $N$ with $\VV_{\bar{\lambda}}\subseteq N$ and an elementary embedding $\map{i}{\bar{M}}{N}$ with $\crit{i}=\bar{\kappa}$ and $i(\bar{\kappa})>\bar{\lambda}$. 
  Since this statement can be formulated by a $\Sigma_1$-formula using the sets $\bar{\kappa}$, $\bar{M}$ and $\VV_{\bar{\lambda}}$ as parameters and all of these parameters are elements of $X$, 
  we can apply the \emph{$\Sigma_1$-Reflection Principle}\footnote{The $\Sigma_1$-Reflection Principle states  that for every uncountable cardinal $\mu$, the set $\HH{\mu}$ is a $\Sigma_1$-elementary submodel of $\VV$ (see {\cite{MR189983}}).} to conclude that  the given statement also holds in $\HH{\bar{\theta}}$.  
 But then we  know that the  statement holds in $X$ and hence the elementarity of $j$ yields a transitive set $N$ with $\VV_\lambda\subseteq N$ and an elementary embedding $\map{i}{M}{N}$ with $\crit{i}=\kappa$ and $i(\kappa)>\lambda$.  
\end{proof}


\section{Ordinal anticipating Laver functions}\label{section:laver}

We now use the results of the previous section to show that strong unfoldability implies the validity of certain non-trivial guessing principles. 
 In the terminology of \cite{Hamkins:LaverDiamond}, the following result proves the existence of \emph{ordinal anticipating Laver function for strong unfoldability}.  
 Remember that a cardinal $\kappa$ is \emph{$\Sigma_2$-reflecting} if it is  inaccessible  and $\VV_\kappa\prec_{\Sigma_2}\VV$ holds. All strongly unfoldable cardinals are $\Sigma_2$-reflecting.\footnote{This follows directly from a combination of Theorem \ref{theorem:Equivalent} and {\cite[Corollary 2.3]{SRminus}}. But the statement can also be directly proven through a small variation of the corresponding argument for strong cardinals (see {\cite[Exercise 26.6]{MR1994835}}).}

\begin{theorem}\label{theorem:OrdinalAntiLaver}
 There exists a class function $\map{\scrL}{\On}{\On}$ such that the following statements hold: 
 \begin{enumerate}
  \item If $\kappa$ is a $\Sigma_2$-reflecting cardinal, then $\mathscr{L}[\kappa]\subseteq\kappa$. 
  
  \item If $\kappa$ is a strongly unfoldable cardinal, $\alpha,\lambda\in\On$ and $M$ is a transitive $\ZF^-$-model of cardinality $\kappa$ satisfying ${\scrL\restriction\kappa}\in M$ and ${}^{{<}\kappa}M\subseteq M$, then there exists a transitive set $N$ with $\VV_\lambda\subseteq N$ and an elementary embedding $\map{j}{M}{N}$ with $\crit{j}=\kappa$, $j(\kappa)\geq\lambda$ and $j({\scrL\restriction\kappa})(\kappa)=\alpha$.   
 \end{enumerate}
\end{theorem}

\begin{proof}
 Let $\lhd$ denote the canonical well-ordering of $\On\times\On$, i.e. we have
 \begin{equation*}
  \begin{split}
   \langle\alpha,\lambda\rangle\lhd\langle\beta,\eta\rangle ~ \Longleftrightarrow ~ & \max(\alpha,\lambda)<\max(\beta,\eta) \\
   &  \vee ~ (\max(\alpha,\lambda)=\max(\beta,\eta) ~ \wedge ~ \alpha<\beta) \\
   &  ~\vee ~ (\max(\alpha,\lambda)=\max(\beta,\eta) ~ \wedge ~ \alpha=\beta ~ \wedge ~ \lambda<\eta)
  \end{split}
 \end{equation*}
 for all $\alpha,\beta,\lambda,\eta\in\On$. 
 We construct the function $\map{\scrL}{\On}{\On}$ by recursion. Assume that we already constructed $\map{\scrL\restriction\gamma}{\gamma}{\On}$ for some ordinal $\gamma$. 
  If $\gamma$ is an inaccessible cardinal and there exists a pair $\langle\beta,\eta\rangle$ of ordinals with the property that there is a transitive $\ZF^-$-model $M$ of cardinality $\gamma$ with ${\scrL\restriction\gamma}\in M$ and ${}^{{<}\gamma}M\subseteq M$ such that  for every transitive set $N$ with $\VV_\eta\subseteq N$, there is no elementary embedding $\map{j}{M}{N}$ with $\crit{j}=\gamma$, $j(\gamma)\geq\eta$ and $j({\scrL\restriction\gamma})(\gamma)=\beta$, then we define $\scrL(\gamma)$ to be the first component of the $\lhd$-least such pair of ordinals.  
  In the other case, if the above assumption fails, then we define $\scrL(\gamma)=0$. 
  This completes the construction of $\scrL$.

  \begin{claim*}
   If $\kappa$ is a $\Sigma_2$-reflecting cardinal, then $\mathscr{L}[\kappa]\subseteq\kappa$. 
  \end{claim*}
  
  \begin{proof}[Proof of the Claim]
   Assume, towards a contradiction, that there exists an ordinal $\gamma<\kappa$ with $\scrL(\gamma)\geq\kappa$. Let $\gamma$ be the minimal ordinal with this property. 
  Then $\gamma$ is an inaccessible cardinal and the function ${\scrL\restriction\gamma}$ is an element of $\VV_\kappa$. 
  Moreover, there exists a pair $\langle\beta,\eta\rangle$ of ordinals with the property that there is a transitive $\ZF^-$-model $M$ of cardinality $\gamma$ with ${\scrL\restriction\gamma}\in M$ and ${}^{{<}\gamma}M\subseteq M$ such that  for every transitive set $N$ with $\VV_\eta\subseteq N$, there is no elementary embedding $\map{j}{M}{N}$ with $\crit{j}=\gamma$, $j(\gamma)\geq\eta$ and $j({\scrL\restriction\gamma})(\gamma)=\beta$. 
  Since the existence of such a pair of ordinals can be expressed by a $\Sigma_2$-formula with parameter ${\scrL\restriction\gamma}$, the fact that $\kappa$ is a $\Sigma_2$-reflecting cardinal allows us to find a pair $\langle\beta_0,\eta_0\rangle$ of ordinals less than $\kappa$ that witnesses the validity of the above statement. 
  But now the definition of $\scrL$ ensures that $\scrL(\gamma)\leq\max(\beta_0,\lambda_0)<\kappa\leq\scrL(\gamma)$, a contradiction.  
  \end{proof}

  Now, assume, towards a contradiction, that   $\kappa$ is a strongly unfoldable cardinal, $\alpha,\lambda\in\On$ and $M$ is a transitive $\ZF^-$-model of cardinality $\kappa$ with ${\scrL\restriction\kappa}\in M$ and ${}^{{<}\kappa}M\subseteq M$ such that for every transitive set $N$ with $\VV_\lambda\subseteq N$, there is no  elementary embedding $\map{k}{M}{N}$ with $\crit{k}=\kappa$, $k(\kappa)\geq\lambda$ and $k({\scrL\restriction\kappa})(\kappa)=\alpha$.   
  Let $\langle\alpha,\lambda\rangle$ denote the $\lhd$-least  pair of ordinals with this property and let $M$ be a transitive $\ZF^-$-model witnessing this. 
 Pick a cardinal $\theta>\max\{\alpha,\kappa,\lambda\}$ that is a limit of fixed points of the $\beth$-function. 
 %
 By Lemma \ref{lemma:ShrewdEmbeddings}, there exist cardinals $\bar{\kappa}<\bar{\theta}<\kappa$, an elementary submodel $X$ of $\HH{\bar{\theta}}$ with $(\bar{\kappa}+1)\cup{}^{{<}\bar{\kappa}}X\subseteq X$  and an elementary embedding $\map{j}{X}{\HH{\theta}}$ with $j\restriction\bar{\kappa}=\id_{\bar{\kappa}}$, $j(\bar{\kappa})=\kappa$ and $\alpha,\lambda,M,\scrL\restriction\kappa\in\ran{j}$. 
 Then elementarity implies that $\bar{\theta}$ is a  limit of fixed points of the $\beth$-function and 
 this guarantees that $\VV_\xi\in\HH{\bar{\theta}}$ for all $\xi<\bar{\theta}$. 
  Pick $\bar{\alpha},\bar{\lambda},\bar{M}\in X$ with $j(\bar{\alpha})=\alpha$, $j(\bar{\lambda})=\lambda$ and $j(\bar{M})=M$. 
  Since $\bar{\theta}$ is a  limit of fixed points of the $\beth$-function, the elementarity of $j$ and the fact that ${}^{{<}\bar{\kappa}}\bar{M}\subseteq\HH{\bar{\theta}}$ ensure that $\bar{M}$ is a transitive $\ZF^-$-model of cardinality $\bar{\kappa}$ with ${\scrL\restriction\bar{\kappa}}\in\bar{M}$ and ${}^{{<}\bar{\kappa}}\bar{M}\subseteq\bar{M}\subseteq X$. 
  Moreover, we know that $\bar{\kappa}$ is inaccessible and $j(\scrL\restriction\bar{\kappa})=\scrL\restriction\kappa$.

  \begin{claim*}
   For every transitive set $N$ with $\VV_{\bar{\lambda}}\subseteq N$, there is no elementary embedding $\map{i}{\bar{M}}{N}$ with $\crit{i}=\bar{\kappa}$, $i(\bar{\kappa})\geq\bar{\lambda}$ and $i({\scrL\restriction\bar{\kappa}})(\bar{\kappa})=\bar{\alpha}$. 
  \end{claim*}
 
  \begin{proof}[Proof of the Claim]
   Assume, towards a contradiction, that there exists such a transitive set and such an embedding. 
   Then the existence of these objects can be formulated by a $\Sigma_1$-formula with parameters $\bar{\alpha}$, $\VV_{\bar{\lambda}}$ and ${\scrL\restriction\bar{\kappa}}$, and all of these parameters are elements of $X$.  
   Since the $\Sigma_1$-Reflection Principle and the elementarity of $X$ in $\HH{\bar{\theta}}$  ensure that $X$ is a $\Sigma_1$-elementary submodel of $\VV$, we can find sets $N$ and $i$ in $X$ that possess the listed properties in $X$. 
   But then the elementarity of $j$ ensures that $j(N)$ is a transitive set with $\VV_\lambda\subseteq j(N)$ and $\map{k=j(i)}{M}{j(N)}$ is an  elementary embedding  with $\crit{k}=\kappa$, $k(\kappa)\geq\lambda$ and $k(\scrL\restriction\kappa)(\kappa)=\alpha$, contradicting our assumptions on $\alpha$, $\lambda$ and $M$. 
  \end{proof}

 \begin{claim*}
  $\scrL(\bar{\kappa})=\bar{\alpha}$. 
 \end{claim*}
 
  \begin{proof}[Proof of the Claim]
   %
   By the previous claim, there is a pair $\langle\beta,\eta\rangle$ of ordinals with the property that there exists a transitive $\ZF^-$-model $M^\prime$ of cardinality $\bar{\kappa}$ with ${\scrL\restriction\bar{\kappa}}\in M^\prime$ and ${}^{{<}\bar{\kappa}}M^\prime\subseteq M^\prime$ such that for every transitive set $N$ with $\VV_\eta\subseteq N$, there is no elementary embedding $\map{i}{M^\prime}{N}$ satisfying $\crit{i}=\bar{\kappa}$, $i(\bar{\kappa})\geq\eta$ and $i({\scrL\restriction\bar{\kappa}})(\bar{\kappa})=\beta$. 
   Let $\langle\beta,\eta\rangle$ denote the $\lhd$-least such pair. 
   Then $\scrL(\bar{\kappa})=\beta$ and the previous claim shows that either $\langle\beta,\eta\rangle\lhd\langle\bar{\alpha},\bar{\lambda}\rangle$ or $\langle\beta,\eta\rangle=\langle\bar{\alpha},\bar{\lambda}\rangle$. In particular, it follows that  $\beta,\eta\leq\max(\bar{\alpha},\bar{\lambda})<\bar{\theta}$.  

   Using the absoluteness of $\Sigma_1$-formulas between $\HH{\bar{\theta}}$ and $\VV$ together with the fact that ${}^{{<}\bar{\kappa}}\bar{\kappa}\in\HH{\bar{\kappa}^+}\in\HH{\bar{\theta}}$ and $\VV_\xi\in\HH{\bar{\theta}}$ for all $\xi<\bar{\theta}$, we now know that, in $\HH{\bar{\theta}}$, the pair $\langle\beta,\eta\rangle$ is $\lhd$-least pair of ordinals with the property that there exists  a transitive $\ZF^-$-model $M^\prime$ of cardinality $\bar{\kappa}$  such that  ${\scrL\restriction\bar{\kappa}}\in M^\prime$,  ${}^{{<}\bar{\kappa}}M^\prime\subseteq M^\prime$  and for every transitive set $N$ with $\VV_\eta\subseteq N$, there is no elementary embedding $\map{i}{M^\prime}{N}$ satisfying $\crit{i}=\bar{\kappa}$, $i(\bar{\kappa})\geq\eta$ and $i({\scrL\restriction\bar{\kappa}})(\bar{\kappa})=\beta$.  
   This shows that, in $\HH{\bar{\theta}}$, the pair $\langle\beta,\eta\rangle$ can be defined by a formula that only uses parameters from $X$ and therefore we know that $\beta$ and $\eta$ are both elements of $X$.  
  But then the elementarity of $j$ implies that, in $\HH{\theta}$,  the pair $\langle j(\beta),j(\eta)\rangle$ is $\lhd$-minimal with the property that there exists  a transitive $\ZF^-$-model $M^\prime$ of cardinality $\kappa$  such that ${\scrL\restriction\kappa}\in M^\prime$, ${}^{{<}\kappa}M^\prime\subseteq M^\prime$ and for every transitive set $N$ with $\VV_{j(\eta)}\subseteq N$, there is no elementary embedding $\map{i}{M^\prime}{N}$ with $\crit{i}=\kappa$, $i(\kappa)\geq j(\eta)$ and $i({\scrL\restriction\kappa})(\kappa)=j(\beta)$. 
   In this situation, our choice of $\theta$ ensures that the given statement also holds in $\VV$. Therefore the minimality of the pairs $\langle\alpha,\lambda\rangle$ and $\langle j(\beta),j(\eta)\rangle$ implies that $j(\beta)=\alpha=j(\bar{\alpha})$ and  we can conclude that $\bar{\alpha}=\beta=\scrL(\bar{\kappa})$. 
 \end{proof}

 By the above claim, we have $j({\scrL\restriction\bar{\kappa}})(\bar{\kappa})=({\scrL\restriction\kappa})(\bar{\kappa})=\bar{\alpha}$ and this shows that the set $M$ and the embedding $\map{j\restriction\bar{M}}{\bar{M}}{M}$ witness that there exists a transitive set $N$ with $\VV_{\bar{\lambda}}\subseteq N$ and an elementary embedding $\map{k}{\bar{M}}{N}$ with $\crit{k}=\bar{\kappa}$, $k(\bar{\kappa})\geq\bar{\lambda}$ and $k({\scrL\restriction\bar{\kappa}})(\bar{\kappa})=\bar{\alpha}$. 
 Since this statement can again be formulated by a $\Sigma_1$-formula using the sets $\bar{\alpha}$, ${\scrL\restriction\bar{\kappa}}$, $\bar{M}$ and $\VV_{\bar{\lambda}}$ as parameters, it  holds in $X$ and therefore the elementarity of $j$ implies that there is a transitive set $N$ with $\VV_\lambda\subseteq N$ and an elementary embedding $\map{k}{M}{N}$ with $\crit{k}=\kappa$, $k(\kappa)\geq\lambda$ and $k({\scrL\restriction\kappa})(\kappa)=\alpha$, contradicting our initial assumption. 
\end{proof}

We now use the above result to show that the assumption that $\VV=\HOD_z$ holds for some subset $z$ of $\kappa$ implies the existence of \emph{Laver functions} for strongly unfoldable cardinals $\kappa$.

\begin{theorem}\label{theorem:LaverHODStrUnfold}
 Let $\kappa$ be a strongly unfoldable cardinal and let $z$ be a subset of $\kappa$.  
 Then there exists a function $\map{\ell_z}{\kappa}{\VV_\kappa}$ such that for every $A\in\HOD_z$, all $\lambda\in\On$ and every transitive $\ZF^-$-model $M$ of cardinality $\kappa$ with $\ell_z\in M$ and ${}^{{<}\kappa}M\subseteq M$, there exists a transitive set $N$ with $\VV_\lambda\subseteq N$ and an elementary embedding $\map{j}{M}{N}$ with $\crit{j}=\kappa$, $j(\kappa)\geq\lambda$ and $j(\ell_z)(\kappa)=A$.  
\end{theorem}

\begin{proof}
 First, note that there exists a $\Sigma_2$-formula $\varphi(v_0,v_1)$ with the property that $\ZFC$ proves  that for every set $y$, the class $\Set{B}{\varphi(B,y)}$ consists of all proper initial segments of the canonical well-ordering of $\HOD_y$ (see the proof of {\cite[Lemma 13.25]{MR1940513}} for details). 
 In particular, we can find a $\Sigma_2$-formula $\psi(v_0,v_1,v_2)$ such that $\ZFC$ proves that $\psi(x,y,\alpha)$ holds if and only if $\alpha$ is an ordinal and $x$ is the $\alpha$-th element of the canonical well-ordering of $\HOD_y$. 
 Since  strongly unfoldable cardinals are $\Sigma_2$-reflecting, it  follows that for all $\gamma<\kappa$, the set  $\HOD_{z\cap\gamma}^{\VV_\kappa}$ is an initial segment  of the canonical well-ordering of $\HOD_{z\cap\gamma}$ of order-type $\kappa$.   
 Let $\map{\scrL}{\On}{\On}$ be the function given by Theorem \ref{theorem:OrdinalAntiLaver} and let $\map{\ell_z}{\kappa}{\VV_\kappa}$ denote the unique function with the property that for all $\gamma<\kappa$, the set $\ell_z(\gamma)$ is the $\scrL(\gamma)$-th element of the canonical well-ordering of $\HOD_{z\cap\gamma}^{\VV_\kappa}$.

 Now, fix $A\in\HOD_z$, $\lambda\in\On$ and a transitive $\ZF^-$-model $M$ of cardinality $\kappa$ with $\ell_z\in M$ and ${}^{{<}\kappa}M\subseteq M$. 
 Let $\alpha$ denote the rank of $A$ in the canonical well-ordering of $\HOD_z$. 
    Pick a  cardinal $\rho>\max\{\alpha,\lambda\}$ with $\VV_\rho\prec_{\Sigma_2}\VV$ and an elementary submodel $M^\prime$ of $\HH{\kappa^+}$ of cardinality $\kappa$ satisfying $\kappa\cup\{M,{\scrL\restriction\kappa}\}\cup{}^{{<}\kappa}M^\prime\subseteq M^\prime$. 
   Our setup now ensures that there exists a transitive set $N^\prime$ with $\VV_\rho\subseteq N^\prime$ and an elementary embedding  $\map{j}{M^\prime}{N^\prime}$ with $\crit{j}=\kappa$, $j(\kappa)\geq\rho$ and $j({\scrL\restriction\kappa})(\kappa)=\alpha$.  
  Then, in $M^\prime$, the function $\ell_z$ has the property that for every $\gamma<\kappa$, the set $\ell_z(\gamma)$ is the $({\scrL\restriction\kappa})(\gamma)$-th element in the canonical well-ordering of $\HOD_{z\cap\gamma}^{\VV_\kappa}$. 
  Since $j({\scrL\restriction\kappa})(\kappa)=\alpha$ and $j(z)\cap\kappa=z$, this shows that, in $N^\prime$, the set $j(\ell_z)(\kappa)$ is the $\alpha$-th element in the canonical well-ordering of  $\HOD_z^{\VV_{j(\kappa)}}$. 
  Moreover, since the $\Sigma_2$-correctness of $\VV_\rho$ in $\VV$ implies that $\HH{\rho}=\VV_\rho$, we can use elementarity to show that, in $N^\prime$, all $\Sigma_1$-statements are absolute between $\VV_\rho$ and $\VV_{j(\kappa)}$. 
   In particular, all  $\Sigma_2$-statements are upwards absolute from $\VV_\rho$ to $\VV_{j(\kappa)}^{N^\prime}$. Set $I=\HOD_z\cap\VV_\rho$. 
    Since $\VV_{j(\kappa)}^{N^\prime}$ is a model of $\ZFC$ and $\VV_\rho\prec_{\Sigma_2}\VV$, we can now conclude that the set $I$ is equal to the  initial segment of order-type  $\rho$ of the canonical well-ordering of  $\HOD_z^{\VV_{j(\kappa)}}$ in  $N^\prime$ . 
    Moreover, the induced well-orderings of $I$ in $\VV$ and $\VV_{j(\kappa)}^{N^\prime}$ agree. 
    In particular, the fact that $\alpha<\rho$ implies that $j(\ell_z)(\kappa)$ is the $\alpha$-th element in the canonical well-ordering of  $\HOD_z$ and therefore $j(\ell_z)(\kappa)=A$. 
  Since  $\VV_\kappa\in M$ and $\VV_\rho\subseteq N^\prime$, we know that $\VV_\lambda\subseteq j(M)$ and therefore the  set $j(M)$ and the map $\map{j\restriction M}{M}{j(M)}$ witness that there exists a transitive set $N$ with $\VV_\lambda\subseteq N$ and an elementary embedding $\map{k}{M}{N}$ with $\crit{k}=\kappa$, $k(\kappa)\geq\lambda$ and $k(\ell_z)(\kappa)=A$.  
%
  %
\end{proof}

 Note that, given a set $z$, the assumptions that $\VV=\HOD_z$ holds  is equivalent to the existence of a well-ordering of $\VV$ of order-type $\On$ that is definable by a formula with parameter $z$. 
  This shows that the above result strengthens {\cite[Theorem 25]{Hamkins:LaverDiamond}} by removing all absoluteness requirements on the used well-ordering of $\VV$ and allowing unbounded subsets of the given strongly unfoldable cardinal as parameters.

 We end this section with a standard argument that shows how  $\Diamond_\kappa(\mathrm{Reg})$-sequences  can be constructed from the Laver functions given by the previous result.

\begin{proof}[Proof of Theorem \ref{theorem:HODdiamonds}]
 Assume that $\kappa$ is a strongly unfoldable cardinal with the property that $\POT{\kappa}\subseteq\HOD_z$ holds for some $z\subseteq\kappa$. 
 Let $\map{\ell_z}{\kappa}{\VV_\kappa}$ denote the function given by Theorem  \ref{theorem:LaverHODStrUnfold} and let $\seq{A_\alpha}{\alpha<\kappa}$ be a sequence with the property that $A_\alpha=\ell_z(\alpha)$ holds for all  $\alpha<\kappa$ with $\ell_z(\alpha)\subseteq\alpha$. 
 Fix a subset $A$ of $\kappa$ and a closed unbounded subset $C$ of $\kappa$. 
 Pick an elementary submodel $M$ of $\HH{\kappa^+}$ of cardinality $\kappa$ with $\kappa\cup\{\ell_z,A,C\}\cup{}^{{<}\kappa}M\subseteq M$. 
 By Theorem \ref{theorem:LaverHODStrUnfold}, there exists a transitive set $N$ and an elementary embedding $\map{j}{M}{N}$ with $\crit{j}=\kappa$ and $j(\ell_z)(\kappa)=A$. 
 Since we now have $\kappa\in j(C)$, $j(\ell_z)(\kappa)=A\subseteq\kappa$ and $j(\ell_z)(\kappa)=A=j(A)\cap\kappa$, elementarity implies that that there exists a regular cardinal $\alpha$ in $C$ with  $\ell_z(\alpha)\subseteq\alpha$ and $A\cap\alpha=\ell_z(\alpha)$.  
 But this allows us to conclude that there exists a regular cardinal $\alpha$ in $C$ with $A_\alpha=A\cap\alpha$. 
\end{proof}


\section{Strong chain conditions}\label{section:FilterLayered}

We now present another application of the results of Section \ref{section:equiv} that deals with strengthenings of the $\kappa$-chain condition of partial orders. Our starting point is the following strong chain condition introduced by Cox in \cite{MR3911105}:

\begin{definition} 
Let $\PPP$ be a partial order and let $\kappa$ be an uncountable regular cardinal. 
 \begin{enumerate}
  \item A suborder $\QQQ$ of  $\PPP$ is  \emph{regular} if the inclusion map preserves incompatibility and maximal antichains in $\QQQ$ are maximal in $\PPP$. 
  
  \item (Cox) The partial order $\PPP$ is \emph{$\kappa$-stationarily layered} if the collection $\mathrm{Reg}_\kappa(\PPP)$ of regular suborders of $\PPP$ of size less than $\kappa$ is stationary\footnote{Here, we use Jech’s notion of stationarity in $\Poti{A}{\kappa}$ (see {\cite[Section 4.1]{MR2768680}}), {i.e.} a subset $S$ of $\Poti{A}{\kappa}$ is \emph{stationary in $\Poti{A}{\kappa}$} if it meets every subset of $\Poti{A}{\kappa}$ which is  $\subseteq$-continuous and cofinal in $\Poti{A}{\kappa}$.} in $\Poti{\PPP}{\kappa}$. 
 \end{enumerate}
\end{definition}

A result of Cox showed that every $\kappa$-stationarily layered partial order $\PPP$ has the $\kappa$-Knaster property (see {\cite[Lemma 4]{MR3911105}), {i.e.} every subset of $\PPP$ of size $\kappa$ contains a subset of size $\kappa$ that consists of pairwise compatible conditions. 
 In particular, $\kappa$-stationary layeredness implies the $\kappa$-chain condition. 
 The main result of \cite{MR3620068} now shows that an uncountable regular  $\kappa$ is weakly compact if and only if the $\kappa$-chain condition implies $\kappa$-stationary layeredness.

 The following strengthening of stationary layeredness, introduced in the proof of the main result of \cite{MR3620068},  suggests the possibility of characterizing large cardinal properties stronger than weak compactness through the equivalence of the $\kappa$-chain condition to some other combinatorial property of partial orders:

  \begin{definition}[\cite{MR3620068}]
  Given an uncountable regular cardinal $\kappa$, a cardinal $\lambda\geq\kappa$ and $\calF\subseteq\POT{\Poti{\lambda}{\kappa}}$,  a partial order $\PPP$ of cardinality at most $\lambda$ is \emph{$\calF$-layered} if $$\Set{a\in\Poti{\lambda}{\kappa}}{s[a]\in\mathrm{Reg}_\kappa(\PPP)}\in\calF$$ holds for every surjection $\map{s}{\lambda}{\PPP}$. 
 \end{definition}

 If $\calF$ is a normal filter on $\Poti{\lambda}{\kappa}$, then {\cite[Lemma 3.3]{MR3620068}} shows that every $\calF$-layered partial order is $\kappa$-stationarily layered. 
 Moreover, if $\kappa$ is weakly compact and $\calF_{wc}$ is the normal filter on $\Poti{\kappa}{\kappa}$ induced by the \emph{weakly compact filter} on $\kappa$ (see \cite{MR0281606}), then {\cite[Lemma 4.1]{MR3620068}} shows that every partial order of cardinality at most $\kappa$ that satisfies the $\kappa$-chain condition is $\calF_{wc}$-layered. 
  Finally, {\cite[Lemma 4.3]{MR3620068}} tells us that if $\kappa$ is $\lambda$-supercompact and $\calU$ is a normal ultrafilter on $\Poti{\lambda}{\kappa}$, then every partial order of cardinality at most $\lambda$ satisfying the $\kappa$-chain condition is $\calU$-layered.

 By considering filters with strong closure properties, the notion of filter layeredness provides natural examples of strong chain conditions that are highly productive. Given infinite cardinals $\nu<\kappa\leq\lambda$ and a function $\map{f}{{}^\nu\lambda}{\lambda}$, we define $$\mathrm{Cl}_\kappa(f) ~ = ~ \Set{a\in\Poti{\lambda}{\kappa}}{f[{}^\nu a]\subseteq a}.$$
  If we now assume that $\kappa$ is regular, $\lambda^\nu=\lambda$ holds, $\mu^\nu<\kappa$ holds for all $\mu<\kappa$ and $\calF$ is a normal filter on $\Poti{\lambda}{\kappa}$ with $\mathrm{Cl}_\kappa(f)\in\calF$ for every function $\map{f}{{}^\nu\lambda}{\lambda}$, then {\cite[Lemma 3.4]{MR3620068}} shows that the class of $\calF$-layered partial orders is closed under $\nu$-support products of length $\lambda$. 
  In particular, this shows that for every normal filter $\calF$ on some $\Poti{\lambda}{\kappa}$, the product of two $\calF$-layered partial orders is again $\calF$-layered. 
    In addition, it should be noted that if either $\kappa=\lambda$ is weakly compact and $\calF=\calF_{wc}$,   or if $\kappa$ is $\lambda$-supercompact and $\calF$ is a normal ultrafilter on $\Poti{\lambda}{\kappa}$, then $\mathrm{Cl}_\kappa(f)\in\calF$ holds for every $\nu<\kappa$ and every function $\map{f}{{}^\nu\lambda}{\lambda}$.

 Given cardinals $\kappa<\lambda$, the results listed above raised the question which large cardinal properties of $\kappa$ are implied by the existence of a normal filter $\calF$ on $\Poti{\lambda}{\kappa}$ for cardinals $\kappa<\lambda$ with the property that all partial orders of size at most $\lambda$ satisfying the $\kappa$-chain condition are $\calF$-layered. 
 More specifically, {\cite[Question 7.4]{MR3620068}} asked whether the existence of a normal filter $\calF$ on $\Poti{\kappa^+}{\kappa}$ with the property that every partial order of size $\kappa^+$ satisfying the $\kappa$-chain condition is $\calF$-layered implies that $\kappa$ is a measurable cardinal. 
  In {\cite[Section 6.1]{MR4082998}},  Cody answered this question in the negative by showing that for every \emph{nearly $\kappa^+$-supercompact} cardinal $\kappa$ (see \cite{MR2989393}), there exists a normal filter $\calF$ on $\Poti{\kappa^+}{\kappa}$ with the property that every partial order of size $\kappa^+$  satisfying the $\kappa$-chain condition is $\calF$-layered.
 This implication answers {\cite[Question 7.4]{MR3620068}}, because a result of Cody, Gitik, Hamkins and Schanker in \cite{MR3372604} shows that the first weakly compact cardinal $\kappa$ can be  nearly $\kappa^+$-supercompact.

 Standard arguments show that a nearly $\kappa^+$-supercompact cardinal $\kappa$ is weakly compact and Jensen's square principle $\square_\kappa$ fails. 
 By {\cite[Theorem 0.1]{MR2499432}}, the existence of  a cardinal with these two properties implies the existence of a sharp for a proper class model with a proper class of strong cardinals and a proper class of Woodin cardinals. 
 This raises the question whether a negative answer to {\cite[Question 7.4]{MR3620068}} has very high consistency strength. 
 In the remainder of this section, we will use the techniques developed earlier to prove the following theorem that shows that this question also has a negative answer:

 \begin{theorem}\label{theorem:StronglyUnfoldableLayered}
  If $\kappa$ is a strongly unfoldable cardinal and $\lambda\geq\kappa$ is a cardinal, then there exists a normal filter $\calF$ on $\Poti{\lambda}{\kappa}$ such that the following statements hold: 
  \begin{enumerate}
   \item Every partial order of cardinality at most $\lambda$ that satisfies the $\kappa$-chain condition is $\calF$-layered. 

   \item If $\nu<\kappa$ is a cardinal and $\map{f}{{}^\nu\lambda}{\lambda}$, then $\mathrm{Cl}_\kappa(f)$ is an element of $\calF$. 
  \end{enumerate}
 \end{theorem}

 Lemma \ref{lemma:ShrewdEmbeddings} motivates the following definition of a filter on $\Poti{\lambda}{\kappa}$ induced by shrewdness:

\begin{definition}
 Given an uncountable regular cardinal $\kappa$ and a cardinal $\lambda\geq\kappa$,  define $\calS_\kappa(\lambda)$ to be the collection of all subsets $A$ of $\Poti{\lambda}{\kappa}$   with the property that there exists a set $z$ such that $j[X]\cap\lambda\in A$ holds for 
 all cardinals $\bar{\kappa}<\bar{\theta}<\kappa\leq\lambda<\theta$ with $z\in\HH{\theta}$, 
 all elementary submodels $X$ of $\HH{\bar{\theta}}$ with $(\bar{\kappa}+1)\cup{}^{{<}\bar{\kappa}}X\subseteq X$ 
 and all elementary embeddings $\map{j}{X}{\HH{\theta}}$ with $j\restriction\bar{\kappa}=\id_{\bar{\kappa}}$,  $j(\bar{\kappa})=\kappa$ and $z\in\ran{j}$.  
 \end{definition}

 The next result shows that the collection $\calS_\kappa(\lambda)$ can also be 
 defined with the help of strongly unfoldable embeddings. The proof of this result refines the constructions made in the proof of Theorem \ref{theorem:Equivalent} by using ideas from  {\cite[Section 2]{MR2279655}} and \cite{MR1133077}.

 \begin{lemma}
  If  $\kappa$ is  an inaccessible cardinal and  $\lambda\geq\kappa$ is a cardinal, then the collection $\calS_\kappa(\lambda)$ consists of all subsets $A$ of $\Poti{\lambda}{\kappa}$ with the property that there exists a set $z$ such that $j[\pi(\lambda)]\in j(\pi(A))$ holds  
  whenever $\theta>\lambda$ is a regular cardinal with $A,z\in\HH{\theta}$, 
  $X$ is an elementary submodel of $\HH{\theta}$ of cardinality $\kappa$ such that  $\kappa\cup\{\kappa,\lambda,A,z\}\cup{}^{{<}\kappa}X\subseteq X$, 
  $\map{\pi}{X}{M}$ is the corresponding transitive collapse, 
  $N$ is a transitive set with $\VV_\theta\subseteq N$, and 
  $\map{j}{M}{N}$  is an elementary embedding with $\crit{j}=\kappa$, $j(\kappa)\geq\theta$ and $j\in N$. 
 \end{lemma}
 
 \begin{proof}
  First, assume that the set $z$ witnesses that some subset $A$ of $\Poti{\lambda}{\kappa}$ is an element of $\calS_\kappa(\lambda)$. 
  Fix a regular cardinal $\nu>\lambda^\kappa$ with $z\in\HH{\nu}$. 
  Now, let $\theta$ be a regular cardinal with $\VV_\nu\in\HH{\theta}$, let $X$ be an elementary submodel of $\HH{\theta}$ of cardinality $\kappa$ with $\kappa\cup\{\kappa,\lambda,\nu,A,z\}\cup{}^{{<}\kappa}X\subseteq X$, let $\map{\pi}{X}{M}$ be the corresponding transitive collapse, let $N$ be a transitive set with $\VV_\theta\subseteq N$ and let $\map{j}{M}{N}$  be an elementary embedding with $\crit{j}=\kappa$, $j(\kappa)\geq\theta$ and $j\in N$.  
  Assume, towards a contradiction, that $j[\pi(\lambda)]\notin j(\pi(A))$. 
  Since $\HH{\nu},X,j\in N$ and $(j\circ\pi)[\HH{\nu}\cap X]\cap j(\pi(\lambda))=j[\pi(\lambda)]$, we know that the cardinals $\kappa\leq\lambda<\nu$, the set $\HH{\nu}\cap X$ and the map $(j\circ \pi)\restriction(\HH{\nu}\cap X)$ witness that, in $N$, there exist 
  cardinals $\bar{\kappa}\leq\bar{\lambda}<\bar{\nu}<j(\kappa)$, 
  an elementary submodel $Y$ of $\HH{\bar{\nu}}$ of cardinality $\bar{\kappa}$ with $\bar{\kappa}\cup\{\bar{\kappa},\bar{\lambda}\}\cup{}^{{<}\bar{\kappa}}Y\subseteq Y$, and
  %
  %
  an elementary embedding $\map{k}{Y}{j(\pi(\HH{\nu}))}$ with the property that $k\restriction\bar{\kappa}=\id_{\bar{\kappa}}$, $k(\bar{\kappa})=j(\kappa)$, $k(\bar{\lambda})=j(\pi(\lambda))$, $j(\pi(z))\in\ran{k}$  and $k[Y]\cap k(\bar{\lambda})\notin j(\pi(A))$. 
  The correctness properties of $X$ and the fact that $\map{j\circ \pi}{X}{N}$ is an elementary embedding now yield  
   cardinals $\bar{\kappa}\leq\bar{\lambda}<\bar{\nu}<\kappa$, 
  an elementary submodel $Y$ of $\HH{\bar{\nu}}$ of cardinality $\bar{\kappa}$ with $\bar{\kappa}\cup\{\bar{\kappa},\bar{\lambda}\}\cup{}^{{<}\bar{\kappa}}Y\subseteq Y$,  and 
  an elementary embedding $\map{k}{Y}{\HH{\nu}}$ with $k\restriction\bar{\kappa}=\id_{\bar{\kappa}}$, $k(\bar{\kappa})=\kappa$, $k(\bar{\lambda})=\lambda$, $z\in\ran{k}$ and $k[Y]\cap\lambda\notin A$. 
  This contradicts the fact that $z$ witnesses that $A$ is an element of $\calS_\kappa(\lambda)$. 
  These computations show that the set $\langle \lambda,\VV_\nu,z\rangle$ witnesses that $A$ is contained in the   collection of subsets of $\Poti{\lambda}{\kappa}$ defined above.

  Now, let $A$ be a subset of $\Poti{\lambda}{\kappa}$ and let  $z$ be a set that witnesses that $A$ is contained in the collection of subsets of $\Poti{\lambda}{\kappa}$ defined in the statement of the lemma. 
  %
  %
  %
  %
  %
  %
  %
  %
  Fix a regular cardinal $\nu>\lambda^\kappa$ with $z\in\HH{\nu}$ and  let $\rho>\nu$ be a limit of fixed points of the $\beth$-function.  
 Assume, towards a contradiction, that there are cardinals $\bar{\kappa}<\bar{\theta}<\kappa\leq\lambda<\theta$ with $\VV_\rho\in\HH{\theta}$, 
 an elementary submodel $X$ of $\HH{\bar{\theta}}$ with $(\bar{\kappa}+1)\cup{}^{{<}\bar{\kappa}}X\subseteq X$ 
 and an elementary embedding $\map{j}{X}{\HH{\theta}}$ with $j\restriction\bar{\kappa}=\id_{\bar{\kappa}}$,  $j(\bar{\kappa})=\kappa$, $\lambda,\nu,A,z\in\ran{j}$ and $j[X]\cap\lambda\notin A$.   
 Pick $\bar{\lambda},\bar{\nu},\bar{A},\bar{z}\in X$ with $j(\bar{\lambda})=\lambda$, $j(\bar{\nu})=\nu$, $j(\bar{A})=A$ and $j(\bar{z})=z$. 
  Since elementarity ensures that $\HH{\bar{\nu}}\in\HH{\bar{\theta}}$, we know that $\HH{\bar{\nu}}\cap X$ is an elementary submodel of $\HH{\bar{\nu}}$ and $\map{j\restriction(\HH{\bar{\nu}}\cap X)}{\HH{\bar{\nu}}\cap X}{\HH{\nu}}$ is an elementary embedding. 
  Note that the regularity of $\nu$ ensures that $\bar{\nu}$ is a regular cardinal and the map $j\restriction(\HH{\bar{\nu}}\cap X)$ is an element of $\HH{\nu}$. 
 In this situation, the sets $\HH{\bar{\nu}}\cap X$ and $\HH{\nu}$ together with the map   $j\restriction(\HH{\bar{\nu}}\cap X)$ witness that there exists 
 an elementary submodel $Y$ of $\HH{\bar{\nu}}$ with $\bar{\kappa}\cup\{\bar{\kappa},\bar{\lambda},\bar{A},\bar{z}\}\cup(\HH{\bar{\nu}}\cap{}^{{<}\bar{\kappa}}Y)\subseteq Y$, 
 a transitive set $N$ with $\VV_{\bar{\nu}}\subseteq N$ and 
 an elementary embedding $\map{i}{Y}{N}$ with $i\restriction\bar{\kappa}=\id_{\bar{\kappa}}$, $i(\bar{\kappa})>\bar{\nu}$, $i[Y\cap\bar{\lambda}]\notin i(\bar{A})$ and $i\in N$. 
 As in the proof of Theorem \ref{theorem:Equivalent}, we can now use the $\Sigma_1$-Reflection Principle to show that this statement also holds in $X$ and hence the elementarity of $j$ yields 
 an elementary submodel $Y$ of $\HH{\nu}$ with $\kappa\cup\{\kappa,\lambda,A,z\}\cup(\HH{\nu}\cap{}^{{<}\kappa}Y)\subseteq Y$, 
 a transitive set $N$ with $\VV_\nu\subseteq N$ and 
 an elementary embedding $\map{i}{Y}{N}$ with $i\restriction\kappa=\id_\kappa$, $i(\kappa)>\nu$,  $i[Y\cap\lambda]\notin i(A)$ and $i\in N$. 
 The regularity of $\nu$ then implies that ${}^{{<}\kappa}Y\subseteq Y$. Let $\map{\pi}{Y}{M}$ denote the corresponding transitive collapse and set $k=i\circ\pi^{{-}1}$. 
 Since $k[\pi(\lambda)]=i[Y\cap\lambda]\notin i(\lambda)$, we can now conclude that  $\map{k}{M}{N}$ is an elementary embedding with $\crit{k}=\kappa$, $k(\kappa)>\nu$, $k[\pi(\lambda)]\notin k(\pi(A))$ and $k\in N$, contradicting our assumptions on $z$.  
 This shows that the set $\langle\VV_\rho,\lambda,\nu,A,z\rangle$ witnesses that $A$ is an element of $\calS_\kappa(\lambda)$.
  \end{proof}

 In the remainder of this section, we prove that the collection $\calS_\kappa(\lambda)$ witnesses that the conclusion of Theorem \ref{theorem:StronglyUnfoldableLayered} holds true.

 \begin{lemma}
  If $\kappa$ is strongly unfoldable and $\lambda\geq\kappa$ is a cardinal, then $\calS_\kappa(\lambda)$ is a normal filter on $\Poti{\lambda}{\kappa}$.
  \end{lemma}
 
\begin{proof}
 First, notice that the collection $\calS_\kappa(\lambda)$ is trivially closed under supersets. 
  Moreover, by Theorem \ref{theorem:Equivalent} and Lemma \ref{lemma:ShrewdEmbeddings}, the strong unfoldability of $\kappa$ directly implies that for every set $z$, there exist cardinals  $\bar{\kappa}<\bar{\theta}<\kappa\leq\lambda<\theta$ with $z\in\HH{\theta}$,  an  elementary submodel $X$ of $\HH{\bar{\theta}}$ with $(\bar{\kappa}+1)\cup{}^{{<}\bar{\kappa}}X\subseteq X$  and an elementary embedding $\map{j}{X}{\HH{\theta}}$ with $j\restriction\bar{\kappa}=\id_{\bar{\kappa}}$,  $j(\bar{\kappa})=\kappa$ and $z\in\ran{j}$. This shows that the empty set is not contained in $\calS_\kappa(\lambda)$. 
 
 Next, assume that for some ordinal $\gamma<\lambda$, there are cardinals $\bar{\kappa}<\bar{\theta}<\kappa\leq\lambda<\theta$,  an  elementary submodel $X$ of $\HH{\bar{\theta}}$ with $(\bar{\kappa}+1)\cup{}^{{<}\bar{\kappa}}X\subseteq X$  and an elementary embedding $\map{j}{X}{\HH{\theta}}$ with $j\restriction\bar{\kappa}=\id_{\bar{\kappa}}$,  $j(\bar{\kappa})=\kappa$ and $\gamma\in\ran{j}$.  
  Then $j[X]\cap\lambda$ is an element of the set $A_\gamma=\Set{a\in\Poti{\kappa}{\lambda}}{\gamma\in a}$ and this shows that the set $\gamma$ witnesses that $A_\gamma$ is contained in $\calS_\kappa(\lambda)$.

 Finally, fix a sequence $\seq{A_\gamma}{\gamma<\lambda}$ of elements of $\calS_\kappa(\lambda)$ and a sequence $\vec{z}=\seq{z_\gamma}{\gamma<\lambda}$ of sets with the property that for all $\gamma<\lambda$, the set $z_\gamma$ witnesses that $A_\gamma$ is an element of $\calS_\kappa(\lambda)$. 
  Assume that there are cardinals $\bar{\kappa}<\bar{\theta}<\kappa\leq\lambda<\theta$,  an  elementary submodel $X$ of $\HH{\bar{\theta}}$ with $(\bar{\kappa}+1)\cup{}^{{<}\bar{\kappa}}X\subseteq X$  and an elementary embedding $\map{j}{X}{\HH{\theta}}$ with $j\restriction\bar{\kappa}=\id_{\bar{\kappa}}$,  $j(\bar{\kappa})=\kappa$ and $\vec{z}\in\ran{j}$.  
  Given $\xi\in X$ with $j(\xi)<\lambda$, we  have $z_{j(\xi)}\in\ran{j}$ and therefore $j[X]\cap\lambda\in A_{j(\xi)}$. This shows that $j[X]\cap\lambda$ is an element of the diagonal intersection $\Delta_{\gamma<\lambda}A_\gamma$. We can conclude that the set $\vec{z}$ witnesses that $\Delta_{\gamma<\lambda}A_\gamma$ is an element of $\calS_\kappa(\lambda)$. This shows that $\calS_\kappa(\lambda)$ is closed under diagonal intersections.   
\end{proof}

 \begin{lemma}
  Given an uncountable regular cardinal $\kappa$ and cardinals $\nu<\kappa\leq\lambda$, we have $\mathrm{Cl}_\kappa(f)\in\calS_\kappa(\lambda)$ for every function $\map{f}{{}^\nu\lambda}{\lambda}$. 
 \end{lemma}
 
 \begin{proof}
  Assume that $\bar{\kappa}<\bar{\theta}<\kappa\leq\lambda<\theta$ are cardinals with $f\in\HH{\theta}$, 
  $X$ is an elementary submodel of $\HH{\bar{\theta}}$ with $(\bar{\kappa}+1)\cup{}^{{<}\bar{\kappa}}X\subseteq X$ 
 and  $\map{j}{X}{\HH{\theta}}$ is an elementary embedding with $j\restriction\bar{\kappa}=\id_{\bar{\kappa}}$,  $j(\bar{\kappa})=\kappa$ and $\nu,\lambda,f\in\ran{j}$.  
  Then $\nu<\bar{\kappa}$ and $j(\nu)=\nu$. Pick $f^\prime,\bar{\lambda}\in X$ with $j(\bar{\lambda})=\lambda$ and $j(f^\prime)=f$. 
  In addition, fix $\map{s}{\nu}{j[\bar{\lambda}]}$ and let $\map{s^\prime}{\nu}{X\cap\bar{\lambda}}$ denote the unique function with $j\circ s^\prime=s$. Then the closure properties of $X$ imply that $s^\prime$ is an element of $X$ and it is easy to see that  $j(s^\prime)=s$. 
  But, now we have $f(s) = j(f^\prime(s^\prime))  \in  j[\bar{\lambda}]$. 
  This shows that $j[\bar{\lambda}]\in\mathrm{Cl}_\kappa(f)$ and hence the set $\langle\nu,\lambda,f\rangle$ witnesses that $\mathrm{Cl}_\kappa(f)$ is contained in $\calS_\kappa(\lambda)$. 
 \end{proof}

 \begin{lemma}
   Given an uncountable regular cardinal $\kappa$ and a cardinal $\lambda\geq\kappa$, every partial order $\PPP$ of cardinality at most $\lambda$ that satisfies the $\kappa$-chain condition is $\calS_\kappa(\lambda)$-layered. 
 \end{lemma}

\begin{proof}
 Fix some surjection $\map{s}{\lambda}{\PPP}$ and assume that there exist cardinals $\bar{\kappa}<\bar{\theta}<\kappa\leq\lambda<\theta$ with $s\in\HH{\theta}$,  an  elementary submodel $X$ of $\HH{\bar{\theta}}$ with $(\bar{\kappa}+1)\cup{}^{{<}\bar{\kappa}}X\subseteq X$  and an elementary embedding $\map{j}{X}{\HH{\theta}}$ with $j\restriction\bar{\kappa}=\id_{\bar{\kappa}}$,  $j(\bar{\kappa})=\kappa$ and $\lambda,s,\PPP\in\ran{j}$.  
 Then there exists a partial order $\QQQ\in X$ with $j(\QQQ)=\PPP$. In this situation, elementarity directly implies that $\QQQ$ satisfies the $\bar{\kappa}$-chain condition. 
 
 \begin{claim*}
  $j[\QQQ\cap X]$ is a regular suborder of $\PPP$. 
 \end{claim*}
 
 \begin{proof}[Proof of the Claim]
  The elementarity of $j$ directly ensures that $j[\QQQ\cap X]$ is a suborder of $\PPP$ and that the corresponding inclusion map preserves incompatibility. 
  Fix a maximal antichain $\calA$ in $j[\QQQ\cap X]$ and set $\calA_0=j^{{-}1}[\calA]\subseteq \QQQ\cap X$. Then elementarity implies that $\calA_0$ is an antichain in $\QQQ$ and, by earlier remarks, it follows that $\calA_0$ has cardinality less than $\bar{\kappa}$. 
  In this situation, the closure properties of $X$ ensure that $\calA_0$ is an element of $X$ and this directly implies that $j(\calA_0)=\calA$. 
  But then $\calA_0$ is a maximal antichain in $\QQQ$, because otherwise elementarity would yield a condition $q\in(\QQQ\setminus\calA_0)\cap X$ with the property that $\calA_0\cup\{q\}$ is an antichain in $\QQQ$ and this would imply that $\calA\cup\{j(q)\}$ is an antichain in $j[\QQQ\cap X]$ that properly extends $\calA$. 
  This allows us to use the elementarity of $j$ to conclude that $\calA=j(\calA_0)$ is also a maximal antichain in $\PPP$.  
 \end{proof}
 
 Now, pick $\bar{\lambda},\bar{s}\in X$ with $j(\bar{\lambda})=\lambda$ and $j(\bar{s})=s$. 
 Then elementarity implies that $\bar{s}$ is a surjection from $\bar{\lambda}$ onto $\QQQ$. 
 Given $q\in\QQQ\cap X$, there exists $\xi\in X\cap\bar{\lambda}$ satisfying $\bar{s}(\xi)=q$ and therefore we know that $$j(q) ~ = ~ s(j(\xi)) ~ \in ~  s[j[X]\cap\lambda].$$ 
 In the other direction, if $\xi\in X\cap\bar{\lambda}$, then $$s(j(\xi)) ~ = ~ j(\bar{s}(\xi)) ~ \in ~ j[\QQQ\cap X].$$ 
 This shows that $s[j[X]\cap \lambda]=j[\QQQ\cap X]$ and we can conclude that $j[X]\cap \lambda$ is contained in the set $R_s=\Set{a\in\Poti{\kappa}{\lambda}}{s[a]\in\mathrm{Reg}_\kappa(\PPP)}$. 
 In particular, we know that the set $\langle\lambda,s,\PPP\rangle$ witnesses that $R_s$ is an element of $\calS_\kappa(\lambda)$. 
\end{proof}

The statement of Theorem \ref{theorem:StronglyUnfoldableLayered} now follows directly from the combination of the above three lemmas.


 \bibliographystyle{plain}
 \bibliography{references}

\end{document}